\documentclass[reqno0,12pt]{amsart}
\pdfoutput=1
\usepackage{bbm}
\usepackage[nodate]{datetime}
\usepackage[hmargin=27mm,top=24mm,bottom=24mm,a4paper]{geometry}
\usepackage{color}
\usepackage{fancyhdr}
\usepackage{amsmath,amssymb,amsfonts,amsthm}
\usepackage{amstext}
\usepackage{amsmath}
\usepackage{amssymb}
\usepackage{enumerate}
\usepackage{amsbsy}
\usepackage{amsopn}
\usepackage{bbm,amsthm}
\usepackage{amscd}
\usepackage{amsxtra}
\usepackage{enumitem}
\usepackage{hyperref}
\usepackage{todonotes}

\newtheorem{theorem}{Theorem}[section]
\newtheorem*{theorem*}{Theorem}
\newtheorem{lemma}{Lemma}[section]

\newtheorem{corollary}{Corollary}[section]

\theoremstyle{remark}

\newtheorem{remark}{Remark}[section]

\newcommand{\CC}{\mathbb{C}}

\newcommand{\EE}{\mathbb{E}}
\newcommand{\PP}{\mathbb{P}}

\newcommand{\mb}{\mathbb}
\newcommand{\mc}{\mathcal}

\newcommand{\leqs}{\leqslant }
\newcommand{\geqs}{\geqslant }

\begin{document}
\title{Sign changes of the partial sums of a random multiplicative function}
\author{Marco Aymone}
\address{ Marco Aymone, Departamento de Matem\'atica, Universidade Federal de Minas Gerais (UFMG), Brazil.}
\email{marco@mat.ufmg.br} 
\author{Winston Heap}
\address{Winston Heap}
\email{winstonheap@gmail.com}
\author{Jing Zhao}
\address{Jing Zhao}
\email{jingzh95@gmail.com}
\begin{abstract}
We provide a simple proof that the partial sums $\sum_{n\leq x}f(n)$ of a Rademacher random  multiplicative function $f$ change sign infinitely often as $x\to\infty$, almost surely.
\end{abstract}
\maketitle

\section{Introduction}

The behaviour of partial sums of multiplicative functions has a long history exhibiting deep connections with arithmetic. The partial sums of the M\"obius function for instance, $M(x)=\sum_{n\leqs x}\mu(n)$, are intimately related to analytic properties of the Riemann zeta function. Indeed, it is well known that the bound $M(x)\ll x^{1/2+\epsilon}$ for all $\epsilon>0$ is equivalent to the Riemann Hypothesis. By standard methods one easily obtains the complementary bound $M(x)=\Omega(\sqrt{x})$ unconditionally. The current gap between these two bounds stands at a factor much smaller than $x^\epsilon$ due to Soundararajan \cite{soundmobius} who showed that $M(x)\ll x^{1/2}\exp((\log x)^{1/2}(\log\log x)^{14})$. The true size of these fluctuations may be yet smaller. Gonek (unpublished) and Ng \cite{Ngmobius} have conjectured that the correct order should be around $x^{1/2}(\log\log\log x)^{5/4}$, although this seems rather deep and out of reach at the moment. A related, but more tractable, problem in this area is that of sign changes of $M(x)$. By relatively simple arguments involving the Dirichlet series of $1/\zeta(s)$, one can show that $M(x)$ changes sign infinitely often as $x\to\infty$\footnote{We say that any function $M(x)$ changes its sign infinitely often as $x\to\infty$ if neither of the two following inequalities holds for all sufficiently large $x$: $M(x)\geq 0$ or $M(x)\leq 0$.}. 


In 1944, Wintner \cite{wintner} introduced a random model for $M(x)$. Let $(f(p))_{p \text{ prime}}$ be a set of \textit{i.i.d.} random variables taking values $\pm1$ with probability $1/2$ (Rademacher variables) and extend this by multiplicativity to $\mb{N}$ with support on the squarefree integers. So for example $f(6)=f(2)f(3)$ and $f(4)=0$. Wintner showed that the bound
\[
\mathcal M (x) := \sum_{n\leqs x}f(n)\ll x^{1/2+\epsilon}
\]     
holds, for all $\epsilon>0$, almost surely. This bound has since seen several improvements, the current best being due independently to  Basquin \cite{basquinn} and Lau--Tenenbaum--Wu \cite{tenenbaum2013} who showed that 
\[
\mathcal M (x)\ll x^{1/2}(\log\log x)^{2+\epsilon}
\]  
almost surely. A very recent result due to Harper \cite{harperomegabounds} shows there almost surely exist arbitrarily large values of $x$ for which 
\[
|\mc{M}(x)|\geqs\sqrt{x}(\log\log x)^{1/4+o(1)}.
\]
Harper has indicated\footnote{Private communication.} that his methods may be modified to produce both large positive and negative values of $\mc{M}(x)$ for arbitrarily large $x$, thus proving that $\mc{M}(x)$ has an infinite number of sign changes. In this note we aim to give a simple proof of this fact akin to the proof in the deterministic case.

\begin{theorem}\label{teorema 1} Let $f$ be a Rademacher random multiplicative function and $\mathcal{M}(x)$ its partial sums up to $x$. Then $\mathcal{M}(x)$ changes its sign infinitely often as $x\to\infty$, almost surely.
\end{theorem}

A classical tool in analytic number theory to analyse questions of sign changes is Landau's oscillation Theorem, which requires that the Dirichlet series $F(s):=\sum_{n=1}^\infty \frac{f(n)}{n^s}$ is analytic in some open set containing its abscissa of convergence. The problem here is that for a Rademacher random multilpicative function $f$, the Dirichlet series $F(s)$ is quite irregular near its abcissa of convergence, which is equal to $1/2$. 
 Our proof of Theorem \ref{theorem infinite sign changes} leverages two facts: that almost surely, $F(\sigma)\to 0$ as $\sigma\to 1/2^+$, and for any real $t\neq 0$,  $\limsup_{\sigma\to1/2^+}|F(\sigma+it)|=\infty$. After partial summation this roughly translates to the almost sure limits 
 \[
 \int_1^\infty \frac{\mc M(x)}{x^{1+\sigma}}dx\to 0,\qquad \int_1^\infty \frac{|\mc M(x)|}{x^{1+\sigma}}dx\to \infty
 \]
as $\sigma\to1/2^+$ which visibly captures the sign changes.

This method allow us to deduce a result relatively stronger than Theorem \ref{teorema 1}. 

\begin{theorem}\label{theorem infinite sign changes}
Let $f$ be a Rademacher random multiplicative function and $\lambda:[1,\infty)\to\mathbb{R}$ be any function  such that $\int_{1}^\infty \frac{|\lambda(u)|}{u^{3/2}}du<\infty$. Then, $\mc{M}(x)+\lambda(x)$ changes sign infinitely often as $x\to\infty$, almost surely. 
\end{theorem}

We can take, for instance, in the nearly extremal case, $\lambda(x)=\pm \frac{\sqrt{x}}{(\log x)(\log\log x)^2}$ and conclude that for these choices $\mc M(x)+\lambda(x)$ changes sign infinitely often, almost surely. As a consequence, we acquire a relatively simple proof of the fact that
$$\mathcal{M}(x)=\Omega_{\pm}\left(\frac{\sqrt{x}}{(\log x)(\log\log x)^2}\right),$$
almost surely. Since $\mathcal{M}(x)$ is always an integer and $f(n)\in\{0,\pm1\}$ for all positive integers $n$, we deduce the following corollary.

\begin{corollary}\label{corolario} For any integer $z$, $\mc M(x)=z$ for an infinite number of integers $x$, almost surely.
\end{corollary}

In probabilistic language, we can view $\mc M(x)$ as a (multiplicative) random walk. We say that a Markov chain is recurrent if it visits at least one of its site infinitely often, almost surely. Otherwise, we say that the Markov chain is transient.  An interesting phenomenon is that the \textit{simple random walk} is recurrent in dimensions $1$ and $2$, but  transient in dimension $3$ and above. Corollary \ref{corolario} can therefore be interpreted as saying that $\mc M(x)$ is recurrent (although strictly speaking $\mc M(x)$ is not a Markov chain due to the dependence structure of $f(n)$).

\section{Proof of the main result}
Clearly Theorem \ref{teorema 1} is a consequence of Theorem \ref{theorem infinite sign changes} by choosing $\lambda(x)=0$ for all $x$. Thus, we will focus in the proof of Theorem \ref{theorem infinite sign changes}.

In what follows, $p$ denotes a generic prime number, and $\sigma$ a parameter bigger than $1/2$. By the Kolmogorov one-series Theorem, the Dirichlet series $\sum_{p}\frac{f(p)}{p^s}$ converges almost surely, provided that $Re(s)=\sigma>1/2$. This result will be implicit in all of the following Lemmas. 
\subsection{Some results for random Dirichlet series with independent summands}
\begin{lemma}\label{lema lei do logaritmo iterado para series de dirichlet}
Let $(f(p))_{p\text{ prime }}$ be Rademacher random variables. Then, for any $\epsilon>0$ there exists a sequence $\sigma_k\to1/2^+$ such that
$$\sum_{p}\frac{f(p)}{p^{\sigma_k}}\ll \left(\log \left(\frac{1}{2\sigma_k-1}\right) \right)^{1/2+\epsilon},$$
almost surely.
\end{lemma}
\begin{proof} The series $\sum_{p}\frac{f(p)}{p^\sigma}$ has variance 
\[
\sum_{p}\frac{1}{p^{2\sigma}}=\log \zeta(2\sigma)+O(1)=\log \left(\frac{1}{2\sigma-1}\right)+O(1).
\]

Hoeffding's inequality for infinite summands (for a proof see e.g.\cite{aymoneLIL}, Lemmas 2.2 and 2.3) states that for any $\lambda>0$
\begin{equation*}
\PP\left(\sum_{p}\frac{f(p)}{p^\sigma}\geq \lambda\right)\leq\exp\left(-\frac{\lambda^2}{2\sum_{p}\frac{1}{p^{2\sigma}}}\right).
\end{equation*}   
Since the random variables $(f(p))_p$ are symmetric, we can replace in the probability above, the infinite sum by its absolute value at a cost of twice the upper bound. Therefore, by choosing $\lambda=\sqrt{2}(\log(1/(2\sigma-1)))^{1/2+\epsilon}$, we obtain that for some constant $c>0$
\begin{equation*}
\PP\left(\left|\sum_{p}\frac{f(p)}{p^\sigma}\right|\geq \sqrt{2} (\log(1/(2\sigma-1)))^{1/2+\epsilon} \right)\leq c\exp\left(-(\log(1/(2\sigma-1)))^{2\epsilon}\right).
\end{equation*}   
Thus, by the Borel-Cantelli lemma, for any sequence $\sigma_k\to1/2^+$ such that
$$\sum_{k=1}^\infty\exp\left(-(\log(1/(2\sigma_k-1)))^{2\epsilon}\right)<\infty,$$
we have that the target bound holds almost surely along this sequence.
\end{proof}
\begin{remark} It is interesting to observe that the upper bound in Lemma \ref{lema lei do logaritmo iterado para series de dirichlet}
holds for all $\sigma$ sufficiently close to $1/2^+$, not just for a sequence. A proof of this can be done by following the steps of \cite{aymoneLIL} from which we can deduce a sharp upper bound at the level of the law of the iterated logarithm. For our purposes, we only need this weaker bound for a sequence $\sigma_k$, we thank Adam Harper for pointing this out.
\end{remark}
\begin{lemma}\label{lemma CLT} Let $(f(p))_{p\text{ prime}}$ be Rademacher random variables. Then, for each fixed $t\in\mathbb{R}\setminus\{0\}$, the infinite sum
$$\sqrt{\frac{2}{\log\left(\frac{1}{2\sigma-1} \right)}}\sum_{p}\frac{f(p)\cos(t\log p)}{p^{\sigma}}$$
converges in probability distribution, as $\sigma\to1/2^+$, to a standard Gaussian random variable.
\end{lemma}
\begin{proof}
Let 
\begin{align*}
S(\sigma)&:=\sum_{p}\frac{f(p)\cos(t\log p)}{p^{\sigma}},\\
V(\sigma)&:=\EE S(\sigma)^2=\sum_{p}\frac{\cos^2(t\log p)}{p^{2\sigma}}. 
\end{align*}
Now, $\cos^2(t\log p)=\frac{1+\cos(2t\log p)}{2}$ 
and 
$$\sum_{p}\frac{\cos(2t\log p)}{2p^{2\sigma}}=\frac{1}{2}Re( \log\zeta(2\sigma+2it))+O(1),$$ which, by the classical estimates for the Riemann zeta function, is $O(1)$ provided that $t\neq 0$ is fixed. Thus, we have that the variance 
$$V(\sigma)=\frac{1}{2}\log\left(\frac{1}{2\sigma-1} \right)+O(1).$$
Now we argue as in the proof of Lemma 3.3 of \cite{aymonerealzeros} to show the convergence to the Gaussian. We do this by the method of characteristic functions.

By the independence of $(f(p))_{p\text{ prime}}$ and the dominated convergence theorem:
\begin{align*}
\varphi_\sigma(s)&:=\EE \exp \left(\frac{is S(\sigma)}{\sqrt{V(\sigma)} }\right)=\prod_{p}\cos\left(\frac{s\cos(t\log p)}{\sqrt{V(\sigma)}p^\sigma}\right).
\end{align*}
Our aim is to show that for each fixed $s\in\mathbb{R}$, $\varphi_\sigma(s)\to \exp(-s^2/2)$ as $\sigma\to 1/2^+$. Observe that $\varphi_\sigma(s)$ is an even function of $s$, so we may assume $s\geq 0$. Also, note that for each $s\geq 0$ we may choose $\sigma>1/2$ such that 
\[
\left|\frac{s\cos(t\log p)}{\sqrt{V(\sigma)}p^\sigma}\right|\leq \frac{1}{100} ,
\,\,\,\,\,\,\,\,\,\,
 0\leq 1-\cos\left(\frac{s\cos(t\log p)}{\sqrt{V(\sigma)}p^\sigma}\right)\leq \frac{1}{100}
\]
for all primes $p$ since $V(\sigma)$ gets large as $\sigma\to1/2^+$.

For $|x|\leq 1/100$, we have that $\log(1-x)=-x+O(x^2)$ and $\cos(x)=1-\frac{x^2}{2}+O(x^4)$. 
Thus, we have:
\begin{align*}
\log \varphi_\sigma(s)&=\sum_{p}\log \cos\left(\frac{s\cos(t\log p)}{\sqrt{V(\sigma)}p^\sigma}\right)\\
&=\sum_{p}\log\left(1-\left(1-\cos\left(\frac{s\cos(t\log p)}{\sqrt{V(\sigma)}p^\sigma}\right)\right)\right)\\
&=-\sum_{p}\left(1-\cos\left(\frac{s\cos(t\log p)}{\sqrt{V(\sigma)}p^\sigma}\right)\right)+O\left(\sum_{p}\left(1-\cos\left(\frac{s\cos(t\log p)}{\sqrt{V(\sigma)}p^\sigma}\right)\right)^2 \right)\\
&=-\sum_{p}\frac{s^2\cos^2(t\log p)}{2V(\sigma)p^{2\sigma}}+O\left(\sum_{p}\frac{s^4}{V^2(\sigma)p^{4\sigma}}\right)\\
&=-\frac{s^2}{2}+O\left(\frac{s^4}{V^2(\sigma)}\right).
\end{align*}
We conclude that $\varphi_\sigma(s)\to \exp(-s^2/2)$ as $\sigma\to 1/2^+$.  
\end{proof}

\begin{lemma}\label{lemma explosao serie de Dirichlet} Let $(f(p))_{p\text{ prime}}$ be Rademacher random variables. Then, for each fixed $t\in\mathbb{R}\setminus\{0\}$, and any sequence $\sigma_k\to1/2^+$, we have that 
$$\limsup_{k\to\infty}\sum_{p}\frac{f(p)\cos(t\log p)}{p^{\sigma_k}}=\infty,$$
almost surely.
\end{lemma}
\begin{proof} Let $S(\sigma)$ and $V(\sigma)$ be as in the proof of Lemma \ref{lemma CLT}. Observe that the event in which
$$\limsup_{k\to\infty}\frac{S(\sigma_k)}{\sqrt{V(\sigma_k)}}\geq 1$$
is a \textit{tail event}, in the sense that it does not depend on the random variables $(f(p))_{p\leq y}$, for any fixed and large $y$.  Now,
\begin{align*}
\PP \left(\limsup_{k\to\infty}\frac{S(\sigma_k)}{\sqrt{V(\sigma_k)}}\geq 1\right)&\geq \PP\left(\frac{S(\sigma_k)}{\sqrt{V(\sigma_k)}}\geq 1,\text{ for infinitely many }k \right)\\
&=\lim_{n\to\infty}\PP\left(\bigcup_{k=n}^\infty\left[\frac{S(\sigma_k)}{\sqrt{V(\sigma_k)}}\geq 1\right]\right)\\
&\geq \frac{1}{\sqrt{2\pi}}\int_{1}^\infty e^{-s^2/2}ds,
\end{align*}
where in the last inequality above we used the convergence to the Gaussian from Lemma \ref{lemma CLT}. Thus, our tail event has positive probability and hence by the Kolmogorov $0-1$ law this probability must be $1$.
\end{proof}

\subsection{Some results for Rademacher random multiplicative functions}
The next result is essentially due to Wintner \cite{wintner}, and for a proof we refer the reader to Lemma 2.1 of \cite{aymonebiased}.
\begin{lemma}\label{lemma log F} Let $f$ be a Rademacher random multiplicative function and $F(s):=\sum_{n=1}^\infty \frac{f(n)}{n^s}$. Then $F(s)$ converges for all $s$ in the half plane $Re(s)>1/2$ and
$$\log F(s)=\sum_{p}\frac{f(p)}{p^s}-\frac{1}{2}\sum_{p}\frac{1}{p^{2s}}+\sum_{p}\sum_{m=3}^{\infty}\frac{(-1)^{m+1}f(p)^m}{mp^{ms}},$$
almost surely.
\end{lemma}
\begin{lemma}\label{lemma integral diverge} Let $f$ be a Rademacher random multiplicative function and $\lambda:[1,\infty)\to\mathbb{R}$ be such that $\int_1^\infty \frac{|\lambda(u)|}{u^{3/2}}du<\infty$. Then, almost surely
$$\int_{1}^\infty\frac{1}{u^{3/2}}\bigg| \lambda(u)+\sum_{n\leq u} f(n)\bigg|du=\infty.$$
\end{lemma}
\begin{proof} Assume, by contradiction, that with positive probability the integral above converges. Then, with positive probability, for each $t\in\mathbb{R}$, the integral 
$$\int_{1}^\infty\frac{1}{u^{3/2+it}} \left(\lambda(u)+\sum_{n\leq u} f(n)\right)du$$
 converges absolutely. Thus, by the dominated convergence theorem
$$\lim_{\sigma\to 1/2^+}\int_{1}^\infty\frac{1}{u^{1+\sigma+it}} \left(\lambda(u)+\sum_{n\leq u} f(n)\right)du = \int_{1}^\infty\frac{1}{u^{3/2+it}} \left(\lambda(u)+\sum_{n\leq u} f(n)\right)du.$$
By the dominated convergence theorem again, the limit
$$\lim_{\sigma\to1/2^+}\int_1^\infty\frac{\lambda(u)}{u^{1+\sigma+it}}du$$
exists and is finite. Thus, the following limit exists and is finite:
$$\lim_{\sigma\to 1/2^+}\int_{1}^\infty\frac{1}{u^{1+\sigma+it}} \left(\sum_{n\leq u} f(n)\right)du. $$
Now we recall that by partial summation the Dirichlet series $F(\sigma+it):=\sum_{n=1}^{\infty}\frac{f(n)}{n^{\sigma+it}}$ satisfies
$$F(\sigma+it)=(\sigma+it)\int_{1}^\infty \frac{1}{u^{1+\sigma+it}}\sum_{n\leq u}f(n) du,$$ 
where this integral converges absolutely for each $\sigma>1/2$ almost surely. Thus we have that for each real $t\neq 0$, the limit $\lim_{\sigma\to 1/2^+}F(\sigma+it)$ exists and is finite with positive probability. Now, by Lemma \ref{lemma log F}, as $\sigma\to 1/2^+$, almost surely we have that 
\begin{align*}
\log F(\sigma+it)&=\sum_{p}\frac{f(p)}{p^{\sigma+it}}-\frac{1}{2}\sum_{p}\frac{1}{p^{2(\sigma+it)}}+\sum_{p}\sum_{m=3}^{\infty}\frac{(-1)^{m+1}f(p)^m}{mp^{m(\sigma+it)}}\\
&=\sum_{p}\frac{f(p)}{p^{\sigma+it}}-\frac{\log \zeta(2\sigma+2it)}{2}+O(1)\\
&=\sum_{p}\frac{f(p)}{p^{\sigma+it}}+O(1),
\end{align*}
provided that $t$ is real, fixed and $\neq 0$. Since $Re \sum_{p}\frac{f(p)}{p^{\sigma+it}}=\sum_{p}\frac{f(p)\cos(t\log p)}{p^{\sigma}}$ and $|F(\sigma+it)|=\exp ( Re \log F(\sigma+it))$, by Lemma \ref{lemma explosao serie de Dirichlet}, almost surely $|F(\sigma+it)|$ tends to $\infty$ along a subsequence $\sigma=\sigma_k\to1/2^+$. This gives the desired contradiction.
\end{proof}
\remark At this point it is interesting to compare our method with that of Hal\'asz \cite{halasz}. Concerning omega bounds for $|\sum_{n\leq x}f(n) |$, Hal\'asz exploited the fact that we can obtain a much larger bound if we consider $$\sup_{1\leq t\leq 2}\sum_p\frac{f(p)\cos(t\log p)}{p^\sigma}.$$ Indeed, he obtained a lower bound for this supremum that is slightly smaller than twice the variance of this random sum evaluated at any fixed $t\neq 0$, with high probability. Later, this idea was exploited by Harper in \cite{harpergaussian} where he improved Hal\'asz's omega bound for the partial sums. In the case of sign changes, we stress that we only need to show that this random sum blows up along a sequence $\sigma_k\to 1/2^+$ for a fixed value of $t\neq 0$, almost surely.

\begin{lemma}\label{lemma F zera perto de meio} Let $f$ be a Rademacher random multiplicative function and $F(\sigma)$ its Dirichlet series. Then, there exists a sequence $\sigma_k\to 1/2^+$ such that $F(\sigma_k)=(2\sigma_k-1)^{1/2+o(1)}$, almost surely.
\end{lemma}
\begin{proof} Indeed, by Lemmas \ref{lema lei do logaritmo iterado para series de dirichlet} and \ref{lemma log F}, for any $0<\epsilon<1/2$, there exists a sequence $\sigma_k\to1/2^+$ such that almost surely:
\begin{align*}
\log F(\sigma_k)&=\sum_{p}\frac{f(p)}{p^{\sigma_k}}-\frac{1}{2}\sum_{p}\frac{1}{p^{2\sigma_k}}+O(1)\\
&=\sum_{p}\frac{f(p)}{p^{\sigma_k}}-\frac{\log \zeta(2\sigma_k)}{2}+O(1)\\
&=O\left( \left(\log \left(\frac{1}{2\sigma_k-1}\right) \right)^{1/2+\epsilon} \right)-\frac{1}{2}\log\left(\frac{1}{2\sigma_k-1}\right)\\
&=(1/2+o(1))\log (2\sigma_k-1).
\end{align*}
\end{proof}
\subsection{Proof of Theorem \ref{theorem infinite sign changes}}
\begin{proof} Assume by contradiction that with positive probability, $\lambda(x)+\sum_{n\leq x}f(n)$ does not changes sign infinitely often as $x\to\infty$, and moreover, say that for some (random) $x_0$, $\lambda(x)+\sum_{n\leq x}f(n)\geq 0$, for all $x\geq x_0$. Thus, for sufficiently large $x$ we have 
$$\int_1^x\frac{1}{u^{3/2}}\left(\lambda(u)+\sum_{n\leq u}f(n) \right)du=O(1)+\int_1^x\frac{1}{u^{3/2}}\left|\lambda(u)+\sum_{n\leq u}f(n) \right|du.$$ 
Now we recall that $F(\sigma):=\sum_{n=1}^\infty \frac{f(n)}{n^\sigma}$ is equal to $\sigma\int_1^{\infty}\frac{1}{u^{1+\sigma}}\sum_{n\leq u}f(n)du$. 
Thus, by Lemma \ref{lemma F zera perto de meio}, along a sequence $\sigma_k\to1/2^+$, we have that with positive probability:
\begin{align*}
(2\sigma_k-1)^{1/2+o(1)}+\sigma_k\int_{1}^\infty\frac{\lambda(u)}{u^{1+\sigma_k}}du&=F(\sigma_k)+\sigma_k\int_{1}^\infty\frac{\lambda(u)}{u^{1+\sigma_k}}du \\
&\geq \sigma_k\int_{1}^x\frac{1}{u^{1+\sigma_k}}\left(\lambda(u)+\sum_{n\leq u}f(n)\right)du.
\end{align*}
Making the limit $\sigma_k\to1/2^+$ in the inequality above, we obtain that
$$O(1)\geq \int_{1}^x\frac{1}{u^{3/2}}\left|\lambda(u)+\sum_{n\leq u}f(n)\right|du,$$
which is by Lemma \ref{lemma integral diverge}, a contradiction for sufficiently large $x$.

Now assume that with positive probability $\lambda(x)+\sum_{n\leq x}f(n)\leq 0$, for all $x\geq x_0$. Arguing as above, we have that
$$\int_1^x\frac{1}{u^{3/2}}\left(\lambda(u)+\sum_{n\leq u}f(n) \right)du=O(1)-\int_1^x\frac{1}{u^{3/2}}\left|\lambda(u)+\sum_{n\leq u}f(n) \right|du.$$
Thus, with positive probability:
\begin{align*}
(2\sigma_k-1)^{1/2+o(1)}+\sigma_k\int_{1}^\infty\frac{\lambda(u)}{u^{1+\sigma}}du&=F(\sigma_k)+\sigma_k\int_{1}^\infty\frac{\lambda(u)}{u^{1+\sigma_k}}du \\
&\leq \sigma_k\int_{1}^x\frac{1}{u^{1+\sigma_k}}\left(\lambda(u)+\sum_{n\leq u}f(n)\right)du.
\end{align*}
Making the limit $\sigma_k\to1/2^+$ in the inequality above, we obtain that
$$O(1)\leq-\int_{1}^x\frac{1}{u^{3/2}}\left|\lambda(u)+\sum_{n\leq u}f(n)\right|du,$$
which is again by Lemma \ref{lemma integral diverge}, a contradiction for sufficiently large $x$. \end{proof}

\section{The multiplicative random walk on $\mathbb{Z}^2$}
As said before, Corollary \ref{corolario} can be interpreted as saying that the multiplicative random walk is recurrent. A natural question could be what happens in dimension $2$. 

By defining $(f(p))_{p\text{ prime}}$ to be \textit{i.i.d.} with $f(p)$ uniformly distributed over the set $\{\pm1, \pm i\}$, and extending $f$ to the positive integers multiplicatively with support on the squarefree integers, then for any positive integer $n$, $f(n)\in\{0,\pm1, \pm i\}$. Therefore, $\mathcal{M}(x)=\sum_{n\leq x}f(n)$ is always a number of the form $a+bi$, where $a$ and $b$ are integers, and hence can be seen as the multiplicative random walk in $2$ dimensions.

Notice that in two dimensions we lose the notion of sign changes, but we can still ask about recurrence. We conclude this section with the following question.

\noindent \textit{Question.} In two dimensions, does the equation $\mathcal{M}(x)=0$ have an infinite number of integer solutions? 

\section{Some simulations}
Below we plot a sample from the multiplicative random walk in dimensions $1$ and $2$.

\begin{figure}[h]
The multiplicative random walk in dimension $1$\\
\includegraphics[scale=0.6]{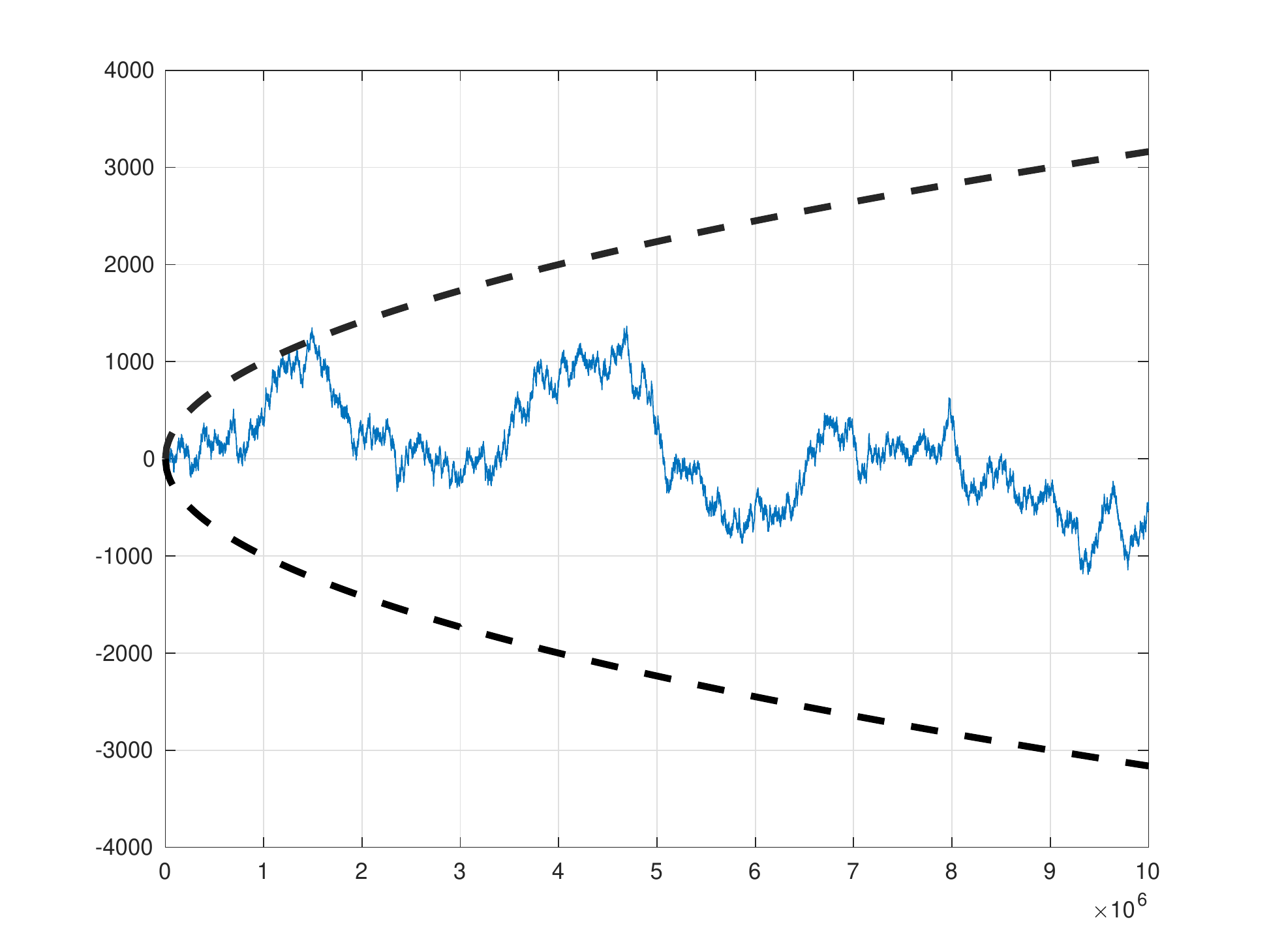} 
\caption{The dashed curve is given by $x\mapsto\pm x^{1/2}$, and the continuous curve is given by the multiplicative random walk in dimension $1$: $x\mapsto\sum_{n\leq x}f(n)$, where $1\leq x\leq 10^7$.}
\end{figure}

\begin{figure}[h]
The multiplicative random walk in dimension $2$\\
\includegraphics[scale=0.6]{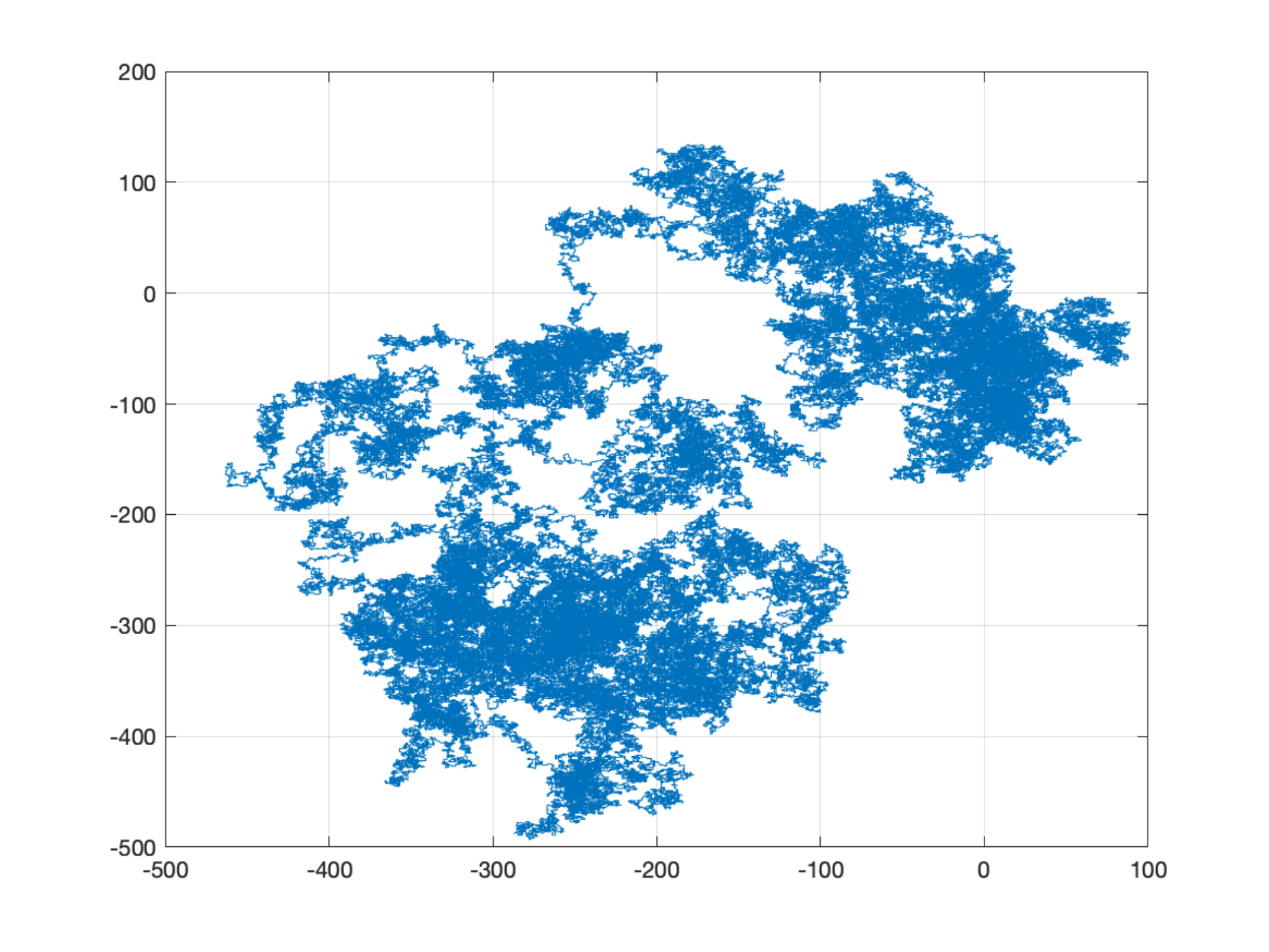} 
\caption{Let $M(x):=\sum_{n\leq x}f(n)$. The blue part is given by set of points visited by the walk, more precisely $\{M(x)\in\CC: 1\leq x\leq 10^6\}$.}
\end{figure}

\newpage
\noindent \textbf{Acknowledgements}. We would like to thank the referee for a careful reading of the paper and for important and useful suggestions and corrections. MA is supported by CNPq, grant Universal no. 403037/2021-2.

\end{document}